\documentclass[reqno,12pt]{amsart}
\usepackage[T1]{fontenc}
\usepackage{amssymb,amsfonts,amsbsy,setspace}
\usepackage[latin1]{inputenc}
\usepackage{setspace}
\makeatletter

 \theoremstyle{plain}
 \newtheorem{thm}{Theorem}[section]
 \newtheorem{lem}[thm]{Lemma}
 
 \numberwithin{equation}{section} 
 \numberwithin{figure}{section} 
 \newtheorem{proposition}[thm]{Proposition}
 \newtheorem{lemma}[thm]{Lemma}
 \theoremstyle{remark} 
 \newtheorem{rmk}[thm]{Remark}
 \newtheorem{exple}[thm]{Example}
 
 \newtheorem*{acknowledgement*}{Acknowledgement}
 \theoremstyle{definition}

\newcommand{\R}{\mathbb R}
\newcommand{\C}{\mathcal{C}}
\newcommand{\N}{\mathbb N}

\newcommand{\M}{\mathcal{M}}
\newcommand{\clomega}{\overline{\Omega}}
\newcommand{\diver}{{\rm div}}

\newcommand{\spt}{{\rm spt}}
\newcommand{\1}{{\bf 1}}
\newcommand{\e}{\varepsilon}
\newcommand{\ws}{\stackrel{*}{\rightharpoonup}}
\newcommand{\Lip}{{\mathcal Lip}}

\newcommand{\D}{\mathcal{D}}

\newcommand{\mut}{{\mu_t}}


\begin{document}

\title[Sandpile in a silos]{Duality theory and optimal transport for sand piles growing in a silos}

\author{Luigi De Pascale}\address{Dipartimento di Matematica, Universit\'a di Pisa, Via Buonarroti 1/c, 56127
  Pisa, ITALY}
\author{Chlo\'e Jimenez}\address{Laboratoire de Math\'ematiques de Bretagne Atlantique, 6 avenue Le Gorgeu CS 93837, 29238 BREST cedex 3,
FRANCE}

\begin{abstract} We prove existence and uniqueness of solutions for a system of PDEs which describes the growth of a sandpile in a silos with flat bottom under the action of a vertical, measure source. The tools we use are a discrete approximation of the source and the duality theory for optimal transport (or Monge-Kantorovich) problems.
\end{abstract}

\date{25 April 2015}
\maketitle

\section{Introduction}

Let $\Omega$ be a bounded, convex open subset of $\R^d$ with $2\le d\in \N$, let $g:\partial \Omega\to \R^+$ be a bounded lower semi-continuous function, and let $f\in L^\infty(0,T,\M(\Omega))$ with  $0 \le f$. We denote by $\M (\Omega)$
the set of Borel measures on $\Omega$ with finite total variation. When $f$ is time constant we will prove existence of solutions $(u,\mu,\nu)$ of problem  
\begin{equation}\label{weakpb}\left\{
\begin{array}{l}
\mbox{(PDE)}\quad \partial_t u-{\rm div}(D_\mut u\mut)=f-\nu\mbox{ in } \R^d\times]0,T[, \\
\mbox{(C): Constraints}\quad\left\{
\begin{array}{l c}
\vert Du\vert \le 1 &\mbox{ in } \Omega\times]0,T[,\\
\vert D_\mut u\vert =1 &\mut\mbox{-a.e. in} \ \Omega\ \mbox{for a.e.} t\in ]0,T[, \\
\end{array}
\right.\\
\\
\mbox{(B): Boundary conditions}\quad \left\{ 
\begin{array}{l }
(B1)\quad 0\le u(x,t) \le g(x) \quad\mbox{ in } \partial\Omega\times]0,T[,\\
(B2)\quad  u(x,t)=g(x)\quad\nu_t\mbox{-a.e.}\!\ (x,t)\in \Omega\times]0,T[,\\
\end{array}
\right.\\
\\
\mbox{(I): Initial conditions}\quad
u(\cdot,0)=0   \mbox{ in }\Omega,\\
\end{array}
\right. 
\end{equation}

with the regularity (\ref{regularity})\footnote{As   $L^{\infty}(0,T,\M^+(\Omega))\subset \M(\Omega\times ]0,T[)$, we will often integrate with respect to $\mu$ or write $\mu$-a.e. $(x,t)$ for $\mu_t$-a.e. $x$, a.e. $t$. }:
\begin{equation}\label{regularity}
\begin{array}{l}
u\in L^\infty(0,T,W^{1,\infty}(\Omega)),\ \partial_t u\in L^\infty(0,T,\M(\Omega)),\\
 \mu\in L^{\infty}(0,T,\M^+(\Omega)),\  \nu\in L^\infty(0,T,\M^+(\partial\Omega)).
 \end{array}
\end{equation}
A function $u$ as above is, in particular, in $\C([0,T];L^2(\Omega))$\footnote{see, for example, \cite{Dro2001Prep}  page 54} 
which gives sense to the initial data ((\ref{weakpb}),(I)). 

In these equations $Du$ denotes the spatial part of the derivative while $\partial_t u$ is the time derivative of $u$.  
The divergence is only spatial and is intended in the distributional sense. The meaning of $D_\mut u$ will be given later and should be seen as the part of $Du$ 
which is relevant for the measure $\mu$. 

As usually the equality given by (PDE) must be intended  as 
$$\frac {d \ }{dt}  \int_{\R^d} u(\cdot,t)\varphi(\cdot) \ dx +\int D_\mut u\cdot D\varphi \ d\mut=\int_{\Omega} \varphi df_t -
\int_{\partial \Omega} \varphi\ d\nu_t, \ \ \mbox{in} \ \D' (0,T)$$
for all $\varphi\in \D(\R^d)$. Here $\D(A)$ denotes the space of smooth functions   compactly supported inside $A$.

The function $u$ will be proved to be unique, we will show that the problem above is equivalent to a variational inequality in the spirit of the original paper \cite{prigozhin1996variational}. One of the main tools will be the duality theory for optimal transport problems. This approach will allow us to study some form of uniqueness and mild regularity for $\mu$. More precisely, we will prove the two following results:

\begin{thm}\label{theo1}
Let $f\in L^\infty(0,T,\M(\Omega))$ and $(u,\mu,\nu)$ satisfying (\ref{regularity}) with the boundary condition (\ref{weakpb},(B1)), and  assume $u(\cdot,t)\in {\mathcal Lip}_1(\Omega)\mbox{ a.e. }t$.\\ Then $(u,\mu,\nu)$ is a solution of (\ref{weakpb}) if and only if (I) is satisfied and
$f_t-\partial_tu(\cdot,t)\in \partial I_\infty(u(\cdot, t))\mbox{a.e. }t\in]0,T[.$ \\
This last condition means that $u$ is a solution, a.e. $t$, of the following maximization problem: 
\begin{equation}\label{Kantor1}
\max\{\langle f-\partial_t u(\cdot, t),v\rangle:\ v\in \Lip_1(\Omega),\ 0\le v(x)\le g(x)\mbox{ on }\partial \Omega\}.
\end{equation}
Moreover (\ref{Kantor1}) is the dual formulation of the mass transport problem (\ref{EqMonge}) and to any optimal choice of $\nu$ corresponds a unique optimal $\mu$ defined by:
 $$\langle\mu_t,\varphi\rangle=\int_{\clomega^2}\int_0^1\varphi((1-s)x+sy)\vert y-x\vert\ ds\ d\gamma_t(x,y)\quad \forall \varphi\in \C_b(\Omega)$$
where $\gamma_t$ (together with $\nu_t$) is any solution of (\ref{EqMonge}).
\end{thm}

\begin{thm}
Let $f\in L^\infty(0,T,\M(\Omega))$ be constant in time. Then (\ref{weakpb}) admits a solution $(u,\mu,\nu)$ satisfying (\ref{regularity}). Moreover $u$ is unique and $\partial_t u\in L^2(\Omega\times]0,T[).$
\end{thm}
\noindent The results above will be proved by approximating $f$ by a finite number of point sources.

The differential system (\ref{weakpb}) has been proposed  several years ago in \cite{prigozhin1996variational} to describe the growth of a sandpile on a bounded 
table under the action of a vertical source here modeled by $f$. 
In the model the sandpile is described as composed by an underlying standing layer, here modeled by $u$ and a rolling layer which is here modeled by $\mu$.
The material rolls downhill only when the standing layer reaches a critical slope which is characteristic of the material. This critical slope, here, is normalised to $1$ and the conditions $(C)$ gives account of this aspect of the behaviour.
Together with the boundary conditions $(BC)$ the system models the growth of the pile inside a silos with wall on $\partial \Omega$ of height $g$.
At some point the sandpile will reach the top of the wall and the sand will start to fall out. The measure $\nu$ describes where this will happen and how much sand will fall out from each point. So at the beginning we expect $\nu_t$ to be $0$ while after some time the sandpile will stabilise and $\partial_t u$ will become $0$ and $\nu$ will have the same mass of $f$.

Allowing $\nu$ to have sign (which is equivalent to assume that an additional source of sand may appear on the boundary) could bring to a loss of uniqueness as shown in the following example.
\begin{exple}
Take $d=1$, $\Omega=]0,1[$, $f=0$, $T=1$ and $g$ constantly equal to $1$. Obviously $(u,\mu,\nu)=(0,0,0)$ is a solution, but we may also take for instance:
$$u(x,t)=(t-x)\1_{\{(x,t):x\le t\}}(x,t),\ \nu_t=-t\delta_{0},$$ 
$$\int_{[0,1]} \varphi(x,t) \ d\mu_t(x)=\int_0^t\int_0^1\varphi((1-s)x)x\  ds dx  \quad \forall \varphi\in \C_b(]0,1[).$$ 
\end{exple}
The choice of $f$ in the space of measures aims to model situations in which the source has dimensions much smaller than the sensitivity of the measure instruments. A possible example from daily life is a hourglass where sometimes the passage for the sand is small at the limit of the imperceptible.

A different approach to this problem is currently pursued by other authors \cite{malusapers}. 
\begin{rmk}
When $f$ is a measure it may happens that  $\mu$ is not better than a measure as shown in several examples in \cite{BoBu2001JEMS,BoBuSe1997CalcVar}
\end{rmk}
As every model, this one is well suited for some situations and it fails in others. An accessible description of several models (included the one we consider), may be found in \cite{hadeler1999dynamical} together with several more references.

Most of the literature is concerned with the Dirichlet case (also known as table problem) with a source $f\in L^p(\Omega)$. In that case the system describes a sandpile growing on a table without walls. At some moment the pile reaches the boundary of the table and the sand start to fall out stabilizing  the pile.
The standard approach consisted in proving that in a {\rm regular} setting the system is equivalent to a variational inequality which may be written as
\begin{equation}\label{varineq}
f-\partial_t u \in \partial I_\infty (u)
\end{equation}
and then proving the existence and uniqueness of a solution $u$ of (\ref{varineq}) for a wide class of $f$. For example, this has been done by Prigozhin in \cite{prigozhin1996variational}  where
existence and uniqueness of $u$   (for the table problem without walls)  is proved under the assumption $f \in (L^4 (0,T, W^{1,4} (\Omega))' $ which is wider than the space we consider. 
Nevertheless, the meaning of $\mu$ is less explicit in that work. Making the link with optimal transport, and making use of the duality theory, allows us to give sense to $\mu$ when $f$ is a measure. 

A similar model for a sandpile growing under the action of a finite number of sources was proposed earlier in \cite{aronsson1972mathematical} and in several unpublished notes by the same author. The ODE arising from that approach was studied in \cite{AroEvaWu1996JDE} and in the same paper $L^p$ approximations  as $p \to \infty$ of the problem were also introduced. The same kind of approximations are used in \cite{evans1997fast} to study the problem for more general sources and to establish a first relationship with the optimal transportation problem.

For the table problem a setting similar to the one of this paper is used in  \cite{IgbidaEvolution2013,DumontDual2009} for a source $f \in L^1(0,T, \M(\Omega))$. The same papers contain some theory for numerical approximations as well as numerical simulations. One of the difficulties one has to face when $f$ is a measure (in the space variable) is due to the lack of regularity of $\mu$. The approach of N. Igbida avoids this problem by working with the flux and introducing some weak formulation of (\ref{weakpb}). To deal with this lack of regularity of $\mu$, we use the tangential calculus with respect to a measure introduced by Bouchitt\'e, Buttazzo and Seppecher in \cite{BoBuSe1997CalcVar}. 

Convergence toward equilibrium in the table problem is studied in \cite{CaCaSi2009CPDE} . Then, stationary solutions  are studied, with or without explicit mention of the sand piles model,  in \cite{BoBuSe1997CRAS, BoBu2001JEMS,CaCa2004JEMS,CaCaCrGi2005CV,CaCaGi2007TAMS,CrMa2007JDE,CrMa2007JDE2,CrMa2012CV}.  
There is still much work to do toward  a complete understanding of the time of convergence to equilibrium when the source is not too regular or controlled from below. The literature on the silos problem is not so rich. Some partial result together with numerics is contained in \cite{CrFV2008NHM}.

Here we choose to start from an empty silos. The theory would be the same (up to technical details) if we assume any initial condition $u_0$ which satisfies $|Du_0| <1$.  If $|Du_0|=1$ on a set of positive measure then one enters the domain of collapsing sand piles  which is different from the one we are considering and is rich of interesting problems (we suggest to start from \cite{evans1997fast} and to follow with the papers in which that paper is cited).

 Finally  a remark on the convexity of $\Omega$. It is clear that the problem would  be interesting also in non-convex domains. Some results for stationary (equilibrium) solutions is contained in the recent paper \cite{CrMa2014Prep} which also consider some anisotropic generalization. Here the assumption of convexity is crucial in section 5 where we use shortest line connecting internal points of $\Omega$ to some point on $\partial \Omega$ determined by the values of $g$.

\medskip

{\it Notations :} The Euclidian norm on $\R^d$ will be denoted by $\vert \cdot \vert$. For any $A$ subset of $\R^d$, $]0,T[$ or $\R^d \times ]0,T[$, we denote by:
\begin{itemize}
\item $\M(A)$ the bounded measures supported in $A$, $\M^+(A)$ being the subset of $\M(A)$ of non-negative, bounded measures,
\item For any vectorial measure $\sigma\in \M(A, \R^n)$, we will denote by $\vert \sigma\vert$ the total variation measure associated to $\sigma$, 
\item $\C(A)$ the set of continuous functions on $A$ while $\C_b(A)$ denotes the set of continuous bounded functions on $A$ equipped with the infinity norm $\Vert\cdot \Vert_\infty$, the smooth compactly supported functions on $A$ will be denoted by
$\D(A):=\C_c^\infty(A)$.
\end{itemize}
For any functional space $\mathcal F$, we will use $\mathcal F'$ for the topological dual of $\mathcal F$. Slightly abusing notations, we will write $\langle\mu, \varphi\rangle_{\M(A), \C_b(A)}$ for $\langle\mu, \varphi\rangle_{(\C_b(A))', \C_b(A)}$ for any $(\mu,\varphi)\in \M(A)\times\C_b(A).$

\section{Tangent space to a measure and integration by parts}
When $u$ is not regular enough the product of $Du$ and a measure $\mu$ does not make, a priori, any sense.
Indeed $Du$ may be not defined on a set which has positive $\mu$ measure.  Here we report few useful notations and results from \cite{BoBuSe1997CalcVar}. This will give sense to $D_{\mu_t}u$ appearing previously. 

Let $\eta\in \M^+(\R^d)$. We can set as in \cite{BoBuSe1997CalcVar}:
$$ X_{\eta}: =\{\psi \in L^2_\eta (\R^d)^d\ : \ \diver (\psi \eta) \in \M(\R^d) \},$$ $$T_{\eta}(x):=\eta-ess\cup\{\psi(x):\psi\in X_{\eta}\},$$
where the divergence is intended in the sense of measures. More precisely $\psi\in L^2_\eta (\R^d)^d$ is in $X_{\eta}$ iff there exists a constant $K$ such that:
$$\displaystyle{\int_{\R^d} D\varphi(x)\cdot \psi(x) d\eta(x) \leq K \|\varphi \|_{\infty}}\ \forall\varphi \in \D (\R^d).$$
For $\varphi\in \D(\R^d)$, the tangential gradient to $\eta$ at $x$ of $\varphi$ is defined as:
$$ D_{\eta}\varphi(x):={P}_{\eta}(x)(D\varphi(x))\ \eta\mbox{-a.e.}\!\ x,\quad\mbox{ with }{P}_{\eta}(x):=\mbox{ orthogonal projector on }{T}_{\eta}(x).$$
As shown in \cite{BoBuSe1997CalcVar}, the operator $u \in \D(\R^d) \mapsto {D}_\eta u \in (L^2_\eta)^d$ can be extended by setting:
$$w=:{D}_\eta v\Leftrightarrow (\exists v_n\in \D(\R^d)\ :\  v_n\to v\mbox{ uniformly, }{D}_{\eta}v_n\rightharpoonup w\mbox{ in }(L^2_\eta)^d).$$
The tangential Sobolev space $H_\eta^1$ is then define as the domain of ${ D}_{\eta}$.
By definition any vector field in $ { X}_{\eta}$ belongs to the dual $H_\eta^{-1}$  and then the following integration by parts formula holds 
\begin{equation}\label{IPPspace}
\int_{\R^d} {D}_{\eta}u(x)\cdot \psi(x)\ d\eta(x)= \langle -\diver (\psi \eta),u \rangle_{H_\eta^{-1},H_\eta^1} \mbox{ for all $u \in H_\eta^1$ and  $\psi \in  {X}_{\eta}$.}
\end{equation}
\begin{exple}\label{exemple} Assume  $\eta\in \M^+(\Omega)$ and let $u \in W^{1,\infty} (\Omega)$. Denote by $v$ any continuous, compactly supported extension of $u$ to $\R^d$. It is easily seen that $v$ belongs to $H^1_\eta$ and that ${D}_\eta v$ will be the same for any other compactly supported extension of $u$. We denote by ${D}_\eta u:={D}_\eta v$.
We have: $|{D}_{\eta}u(x)| \leq \|Du\|_\infty\quad \eta\mbox{-a.e.}x$. Moreover,  (\ref{IPPspace}) rewrites as:
$$\int_{\Omega} { D}_{\eta}u(x)\cdot \psi(x)\ d\eta(x)= \langle -\diver (\psi \eta),u \rangle_{\M(\Omega),\C_b(\Omega)}   \mbox{ for all $\psi \in  { X}_{\eta}$.}$$
\end{exple}

\begin{rmk}
In order to build some notion of space-time tangential gradient, we could have chosen some notion of space-time tangential gradient, to set for any $\mu\in L^\infty(]0,T[,\M(\R^d))$:
$$ {\chi}_\mu =\{\psi \in L^2_\mu (\R^d \times ]0,T[)^d\ : \ \diver (\psi \mu) \in  L^\infty( 0,T,\M(\R^d)) \}.$$
Nevertheless, the result will be exactly the same. Indeed, for any $\psi \in L^2_\mu (\R^d \times ]0,T[)^d$, the following equivalence holds:
$$\psi \in {\chi}_\mu \Leftrightarrow \psi(\cdot, t)\in X_{\mu_t} \mbox{ for a.e.}\!\ t\in]0,T[.$$
Let us prove this equivalence. Assume  that $\psi(\cdot, t)\in X_{\mu_t} \mbox{ for a.e.}\!\ t\in]0,T[$ and take $\varphi\in \D(\R^d)$,  $h\in \D(]0,T[)$ then:
\begin{eqnarray*}
&&\quad \int_0^T\int_{\R^d} D(h\varphi)(x,t)\cdot \psi(x,t) d\mu(x,t)= \int_0^T h(t)\int_{\R^d} D\varphi(x)\cdot \psi(x,t) d\mu_t(x) dt\\
&&\qquad\leq K \int_{0}^T h(t)\|\varphi(\cdot)\|_{\infty}\ dt \leq K \int_{0}^T \Vert h(t)\ \varphi(\cdot)\|_{\infty}\ dt.
\end{eqnarray*}
This means exactly that $\psi \in {\chi}_\mu$.
The other implication is straightforward.
\end{rmk}

\section{Duality and optimal transport }
The results contained in \cite{BoBu2001JEMS} (see also \cite{EvaGan1999}) suggests
to consider the maximization problem (\ref{max}) defined below. Its link with our problem will appear clearly in Theorem \ref{DualEXT}. 
The set ${\mathcal Lip}_1(\Omega)$ is defined by
$$\Lip_1(\Omega):=\left\{v\in \Lip(\R^d):\ v(y)-v(x)\le \vert x-y\vert\mbox{ in } \clomega\times\clomega\right\}.$$
Following \cite{BoBu2001JEMS}, we prove Proposition \ref{dual1}, Theorem \ref{DualEXT} and Proposition \ref{Monge}.
\begin{proposition}\label{dual1}
Let $u\in \Lip_1(\Omega)$ and $\rho \in \M(\Omega)$. 
Then the following extremal values coincide:
\begin{equation}\label{max}
\max\{\langle \rho, v \rangle_{\M(\Omega),\C_b(\Omega)}\ : \ v\in \Lip_1 (\Omega), \ 0\le v\le g \ \mbox{on}\ \partial \Omega\}
\end{equation}
\begin{equation}\label{mincurrent}
\inf_{\sigma\in \M(\Omega,\R^d),\nu\in \M(\partial \Omega)} \left\{ \int_{\Omega} d|\sigma|+\int_{\partial \Omega} g d\nu^+
:\ -{\rm div} \sigma =\rho-\nu\mbox{ in }\R^d \right\}.
\end{equation}
\end{proposition}
Note that any sequence $(v_n)_n$ of admissible applications for (\ref{max}) is uniformly bounded on $\partial \Omega$ and as it is in   $\Lip_1 (\Omega)$, it is uniformly bounded in $\clomega$. The existence of a maximizer is then easily proved.
\begin{proof}
Let us introduce for any $(p,q)\in{\mathcal C}(\partial\Omega)^2 $ the following perturbation functional:
$$H(p,q):=-\sup\{\langle \rho, v \rangle_{\M(\Omega),{\mathcal C}_b(\Omega)}\ : \ v\in \Lip_1 (\Omega), \ v+p\le g,\ 0\le v+q \ \mbox{on}\ \partial \Omega\}.$$
Now compute the Fenchel transform of $H$ on a couple $(p^*,q^*)\in \M(\partial \Omega)^2$:
\begin{eqnarray*}  
&&H^*(p^*,q^*)= \sup_{p,q\in{\mathcal C}(\partial\Omega)}\{\langle p^*,p\rangle_{\M(\partial\Omega),\C(\partial\Omega)} +\langle q^*,q\rangle_ {\M(\partial\Omega),\C(\partial\Omega)}-H(p,q)\}\\ 
&&=\sup_{p,q}\sup_v\left\{  \langle p^*,p\rangle+\langle q^*,q\rangle+\langle \rho,v\rangle:\ v\in{\Lip}_1(\Omega),\ -q\le v,\ p\le g-v\mbox{ on } \partial\Omega\right\}.\\
&&
=\sup_v\left\{ \sup_{p,q}\{ \langle p^*,p\rangle +\langle q^*,q\rangle:\     -q\le v,\ p\le g-v\mbox{ on }\partial\Omega\} +\langle \rho,v\rangle:\   v\in{\Lip}_1(\Omega)\right\}\\
&&=\left\{ \begin{array}{c c}
 \sup_{v,p} \left\{ \langle \rho-q^*,v\rangle +\langle p^*,p\rangle :\   v\in{ \Lip}_1(\Omega),\  p\le g-v \right\}& \mbox{ if } q^*\le 0; \\
+\infty & \mbox{elsewhere.}
\end{array}
\right.
\end{eqnarray*}
It is easily seen that $H^*(p^*,q^*)\not=+\infty$  only if  $ p^*\ge 0$. Let us assume this condition is satisfied. We have:
$$H^*(p^*,q^*)
\le
 \langle p^*,g\rangle+\sup_v \left\{ \langle \rho-p^*-q^*,v\rangle:\   v\in{ \Lip}_1(\Omega)\right\}.$$
Let us show the opposite inequality. Let $(g_n)_n$ a sequence in  $\C(\partial\Omega)$ converging to $g$ at any point of $\partial\Omega$ with $g_n\le g$. Then, by taking $p=g_n-v$ we have:
$$
H^*(p^*,q^*)\ge
\lim_{n\to +\infty}\int_{\partial \Omega} g_n(x)\  dp^*(x)+\sup_v \left\{ \langle \rho-p^*-q^*,v\rangle:\   v\in{ \Lip}_1(\Omega)\right\}.$$
Finally the equality follows using Fatou's lemma and we have:
$$
H^*(p^*,q^*)
=\left\{ \begin{array}{l }
 \langle p^*,g\rangle+\sup_v \left\{ \langle \rho-p^*-q^*,v\rangle:\   v\in{ \Lip}_1(\Omega)\right\} \quad \mbox{ if }p^*\ge 0\mbox{ and } q^*\le 0; \\
+\infty \quad \mbox{elsewhere.}
\end{array}
\right.
$$
And, by standard duality (see \cite{BoBu2001JEMS}):
$$H^*(p^*,q^*)=\left\{ \begin{array}{l }
 \langle p^*,g\rangle+\inf_{\sigma\in \M(\Omega,\R^d)}\left\{\displaystyle{ \int_{\Omega} d |\sigma| : \ -\diver \sigma= \rho-p^*-q^*\mbox{ in }\R^d}\right\}\\
\qquad \mbox{ if } p^*\ge 0\mbox{ and } q^*\le 0; \\
+\infty \mbox{ elsewhere. }
\end{array}
\right.
$$
It can be easily proved that  $H$ is convex. Let us check that it is l.s.c. Let $(p_n,q_n)\in \C(\partial \Omega)^2$ converging uniformly to $(p,q)\in \C(\partial \Omega)^2.$ For any $\e>0$ take $v_n\in \Lip_1(\Omega)$ $\e$-optimal for $H(p_n,q_n)$ that is such that $H(p_n,q_n)\ge -\langle \rho, v_n\rangle -\e.$
Possibly extracting a subsequence, we may assume:
$$\liminf_{n\to +\infty} H(p_n,q_n)=\lim_{n\to +\infty} H(p_n,q_n).$$
From the regularity of $v_n$ and the bounds on $\partial \Omega$, we get:
$$v_n(x)\le \sup_{y\in \partial \Omega} \{\vert y-x\vert +g(y)-p_n(y)\}\le {\rm diam}(\Omega) +\Vert g-p_n\Vert_\infty,$$
$$v_n(x)\ge \inf_{y\in \partial \Omega} \{-\vert y-x\vert-q_n(y)\}\ge - {\rm diam}(\Omega) -\Vert q_n\Vert_\infty.$$
 As $(v_n)_n$ is an equicontinuous and bounded sequence, by Ascoli, a subsequence of $(v_{n_k})_k$ of $(v_n)_n$ converges uniformly to some $v\in \Lip_1(\Omega)$ admissible for $H(p,q)$. Then:
 $$ \liminf_{n\to +\infty} H(p_n,q_n)\ge\lim_{k\to +\infty} -\langle \rho, v_{n_k}\rangle -\e
 = -\langle \rho, v\rangle -\e \ge H(p,q)-\e.$$
 By sending $\e$ to $0$, we get the lower semi-continuity of $H$. The result then follows from the equality $H(0,0)=(H^*)^*(0,0).$
\end{proof}

\begin{thm}\label{DualEXT} Let $\rho\in\M(\Omega)$.\\
\noindent (i) Assume that $(u,\mu,\nu) \in \Lip_1 (\Omega) \times \M^+ (\Omega)\times\M(\partial \Omega)$ is a solution of 
\begin{equation}\label{extremality}
 - \diver (\mu {D}_\mu u)= \rho-\nu \ \mbox{in}\ \R^d,\quad  \vert {D}_\mu u(x) \vert=1\mbox{ $\mu$-a.e.}x,
\end{equation}
\begin{equation}\label{bounds}
\mbox{ with }\left\{ \begin{array}{l}
0\le u\le g\mbox{ on }\partial \Omega,\\
u(x)=0\quad\nu^-\mbox{-a.e.}\!\ x,\quad u(x)=g(x)\quad\nu^+\mbox{-a.e.}\!\ x.
\end{array}
\right.
\end{equation}
Then  $u\in argmax(\ref{max}) $  and, setting   $\sigma={D}_\mu u \mu$, we have $(\sigma,\nu) \in argmin(\ref{mincurrent}).$\\

\noindent (ii) Moreover, if  $u\in \Lip_1 (\Omega)$ and $(\sigma,\nu )\in \M(\clomega)^d\times\M(\partial \Omega)$
are optimal solutions of (\ref{max}) and (\ref{mincurrent}) then setting $\mu=\vert \sigma \vert$, we have:
\begin{equation}\label{gradext}
\left\{\begin{array}{l}
\sigma={D}_\mu u\mu \mbox{ and }
\vert {D}_\mu u\vert=1\quad \mu-\mbox{a.e},\\
u=0\quad\nu^-\mbox{-a.e.}\!\ x,\quad u(x)=g(x)\quad\nu^+\mbox{-a.e.}\!\ x.
\end{array}
\right.
\end{equation}
\end{thm}

\begin{proof}
(i)\ With the assumptions above, $u$  and $(\sigma={D}_\mu u\mu,\nu)$ are admissible for (\ref{max}) and (\ref{mincurrent}).
An integrating  by parts (see (\ref{IPPspace})) leads to:
$$
\begin{array}{rcl}
\langle\rho, u\rangle_{\M(\Omega),\C_b(\Omega)}=\langle -\diver (\mu {D}_\mu u)+\nu,u\rangle_{\M(\clomega),\C(\clomega)}&=&\displaystyle{\int_{\Omega} \vert{D}_\mu u\vert^2\ d\mu +\int_{\partial \Omega} u\ d\nu}\\
\\
&=&\displaystyle{\int_{\Omega}  d\vert\sigma\vert+ \int_{\partial \Omega} g\ d \nu^+.}
\end{array}
$$
By Proposition \ref{dual1}, this implies the result.

(ii) By Proposition \ref{dual1}: %
$$
\langle\rho, u\rangle_{\M(\Omega),\C_b(\Omega)}=\int_{\Omega} d\mu+\int_{\partial\Omega} g(x)\ d\nu^+(x)$$
this implies
$$\displaystyle{\langle-\diver (\sigma)+\nu, u\rangle_{\M(\clomega),\C(\clomega)}}=\int_{\Omega} d\mu+\int_{\partial\Omega} g(x)\ d\nu^+(x),$$
and $$\langle-\diver (\sigma), u\rangle_{\M(\clomega),\C(\clomega)}-\int_{\Omega} d\mu=\int_{\partial \Omega} g(x)\ d\nu^+(x)-\int_{\partial \Omega} u(x)\ d\nu(x).$$
Finally, integrating again by parts (see again (\ref{IPPspace})): 
\begin{equation}\label{etapeExt}
\int_{\Omega} ({D}_\mu u(x)\cdot {d \sigma\over d\mu}(x)-1)\ d\mu(x)=\int_{\partial \Omega} (g(x)-u(x))d\nu^+(x)+\int_{\partial \Omega} u(x) d\nu^-(x).
\end{equation}
To conclude, we notice that
on the one hand as $u$ is in $\Lip_1 (\Omega)$ $${D}_\mu u(x)\cdot {d \sigma\over d\mu}(x)-1 \le 0,\ \mu\mbox{-a.e.}x$$
 on the other hand,  on $\partial \Omega$ $$g(x)-u(x)\ge 0\mbox{ and }u(x)\ge 0.$$ 
This inequalities, combined with (\ref{etapeExt}) give (\ref{gradext}).
\end{proof}

The following proposition shows that (\ref{max}) can be seen as a variant of the dual formulation of the classical Monge mass transportation problem (see \cite{Amb2003} or \cite{Vil2003} for instance):
\begin{proposition}\label{Monge}
Let $\rho \in \M(\Omega)$.
 Then the extremal value (\ref{max})   coincides with the following extremal one:
\begin{multline}\label{dualplan}
\hspace{-1cm}\min_{\gamma \in \M^+ (\clomega \times \clomega), \nu\in \M(\partial \Omega)} \left\{ \int_{\clomega \times \clomega}  \vert x-y\vert d\gamma(x,y)+\int_{\partial \Omega}g d \nu^+\ | 
  \ \pi^1_\sharp \gamma= \rho^- +\nu^+, \  \pi^2_\sharp \gamma= \rho^++\nu^-\right\}.
\end{multline}

\end{proposition}
This is a consequence of classic duality in the $L^1$ theory of optimal transport (see for instance \cite{BoBu2001JEMS}) which implies the following equality:
\begin{equation}\label{was}
\min(\ref{dualplan})=\min_{\nu\in \M(\partial \Omega)}\{W_1(\rho^-+\nu^+,\rho^++\nu^-)+\int_{\partial \Omega} g\ d\nu^+\}=\inf(\ref{mincurrent}).
\end{equation}
 Where $W_1$ is the Wasserstein distance:
 \begin{multline*}
 W_1(\rho^-+\nu^+,\rho^++\nu^-):=\\
 \min_{\gamma \in \M^+ (\clomega \times \clomega)} \left\{ \int_{\clomega \times \clomega}  \vert x-y\vert d\gamma(x,y): 
  \ \pi^1_\sharp \gamma= \rho^- +\nu^+, \  \pi^2_\sharp \gamma= \rho^++\nu^-\right\}.
 \end{multline*} 
 %
 We have the following result:
 \begin{lemma}\label{nu-nu}\label{PrimalDualClassic}
 \begin{itemize}
\item[(i)] Let $(\gamma,\nu)\in \M(\clomega\times \clomega)\times \M(\partial \Omega)$ be a solution of (\ref{dualplan}), then it exists a unique $\sigma\in \M(\Omega, \R^d)$ such that $(\sigma,\nu)$ is a solution of (\ref{mincurrent}):
$$\langle\sigma,\phi \rangle_{\M(\Omega, \R^d),\C_b(\Omega,\R^d)}:=\int_{\Omega^2}\int_0^1 \phi((1-s)x+sy )\cdot(y-x)\ ds d\gamma(x,y)\quad \forall \phi\in \C_b(\Omega,\R^d).$$
On the contrary, if $(\sigma,\nu)$ is a solution of (\ref{mincurrent}), it exists $\gamma$ such that $(\gamma,\nu)$ is optimal for (\ref{dualplan}) and the previous expression of $\sigma$ holds.
\item[(ii)] Let $u$ optimal for (\ref{max}) and $\sigma$ as in (i). Then $u$ is differentiable $\vert \sigma \vert$-almost everywhere and: $$D_{\vert \sigma \vert} u(x)=Du(x)\quad \vert \sigma\vert\mbox{-a.e. }x.$$
\item[(iii)] Let $(\gamma,\nu,u)$ admissible for (\ref{dualplan}) and (\ref{max}), then $(\gamma,\nu,u)$ are optimal if and only if:
 $$\left\{
 \begin{array}{l}
 u(y)-u(x)=\vert x-y\vert\quad \gamma-\mbox{a.e. }(x,y)\\
u(x)=0\quad\nu^-\mbox{-a.e.}\!\ x,\quad u(x)=g(x)\quad\nu^+\mbox{-a.e.}\!\ x.
 \end{array}
 \right.
 $$
 \end{itemize}
 \end{lemma}
\begin{proof} Point (i), using (\ref{was}), is a consequence of Theorem 4.6. of \cite{BoBu2001JEMS}. Uniqueness of $\sigma$ can found in \cite{AmbPra2003}. Point (iii) gives a Primal-Dual optimality condition for (\ref{dualplan}) and (\ref{max}), it is a  corollary of Proposition \ref{Monge}, it is very classic in the usual Monge case (see \cite{Amb2003}, Corollary 2.1.). 
The differentiability of $u$ $\vert \sigma \vert$-a.e. is a well known result in the $L^1$ theory of optimal transport, moreover it has been proved that (see for instance \cite{EvaGan1999} for both properties):
$$ (z\in ]x,y[,\ u(y)-u(x)=\vert y-x\vert)\Rightarrow Du(z)={y-x\over \vert y-x\vert}.$$  Then by (i) and (iii), ${d\sigma\over d\vert \sigma \vert}=Du$ and so by Theorem \ref{DualEXT} $Du=D_{\vert \sigma \vert} u$ $\vert \sigma\vert$-a.e. \end{proof}
\begin{rmk}\label{PrimD}
In addition to point (iii), we have:
$$\langle \rho, u\rangle \le \int \vert y-x\vert\ d\gamma(x,y) +\int g\ d\nu^+$$
for any admissible $(u,\gamma,\nu)$ for (\ref{max}) and (\ref{dualplan}). Equality implies optimality of $u$ and $(\gamma,\nu)$.
\end{rmk}
Let $\gamma$ be an optimal plan for (\ref{dualplan})  (and let $\nu$ the corresponding boundary measure) we can decompose $\gamma$ in four parts according to the origin and the destination of the mass transported
$$ \gamma=\gamma_{ii }+\gamma_{bi }+\gamma_{ib }+\gamma_{bb}$$ 
$$\mbox{ with } \gamma_{ii }=\gamma_{| \Omega \times \Omega}, \  \gamma_{bi }=\gamma_{| \partial \Omega \times \Omega}, \  \gamma_{ib }=\gamma_{| \Omega \times \partial \Omega}, \  \gamma_{bb }=\gamma_{| \partial \Omega \times \partial \Omega},$$
 where $i$ is for interior and $b$ is for boundary.  We have $\nu^+= \pi^1_\sharp (\gamma_{bi}+ \gamma_{bb})$ and $\nu^-= \pi^2_\sharp (\gamma_{ib}+ \gamma_{bb})$, then:
\begin{multline*}
 \int_{\clomega \times \clomega} \vert x-y \vert d\gamma(x,y)+\int_{\partial \Omega}g d \nu^+= \int_{\Omega \times \Omega} \vert x-y\vert d \gamma_{ii}  +
 \int_{\partial\Omega \times \Omega} \{\vert x-y\vert +g(x)\} d \gamma_{bi} \\
 + \int_{\Omega \times \partial\Omega} \vert x-y\vert d \gamma_{ib}+  
 \int_{\partial\Omega \times \partial\Omega} \{\vert x-y\vert+g(x) \} d \gamma_{bb}. 
  \end{multline*}

Note that
$ \int_{\partial\Omega \times \partial\Omega} \{\vert x-y\vert +g(x)\} d \gamma_{bb}(x,y)\geq 0$
and  $\gamma-\gamma_{bb}$ is still admissible for (\ref{dualplan}). Then since $\gamma$ is minimizing we have 
$$ \int_{\partial\Omega \times \partial\Omega} \{\vert x-y\vert+g(x) \} d \gamma_{bb}(x,y)=0.$$
From now on we will assume that
$\gamma_{bb}=0$
(Otherwise $\spt(\gamma_{bb})\subset\{(x,x):\ g(x)=0\}$). 
The minimizing properties of $\gamma$ and the duality of Proposition \ref{Monge} permits to characterize the points in the supports of $\gamma_{bi}$ and $\gamma_{ib}$.
\begin{proposition}\label{projection} 
We define the following multivalued maps:
$$p_g^+(y)=\{z\in \partial\Omega:\ \vert y-z\vert+g(z)\le \vert y-x\vert+g(x)\ \forall x\in \partial \Omega\},$$
$$p^-(x)=\{\omega\in \partial \Omega:\ \vert x-\omega\vert \le \vert x-y\vert\ \forall y\in \partial \Omega\}.$$
If $(x,y) \in \spt \gamma_{bi} $ then $x \in p^+_g(y)$, if $(x,y) \in \spt \gamma_{ib} $ then $y \in p^-(x)$.
\end{proposition}
\begin{proof} We prove the statement about $\gamma_{bi}$, the other being similar. The part of the cost we look at is 
$$ \int_{\partial \Omega \times \Omega} \{\vert x-y\vert+g(x)\}\  d \gamma_{bi}(x,y).$$
By definition of $p_g^+$ we have that $\vert z-y\vert+g(z)\leq \vert x-y\vert+g(x)$  $\gamma_{bi}$-a.e.$(x,y)$ for all $z\in p_g^+(y)$.
Then, for all measurable selection $s$ of $p_g^+$ if we replace $\gamma_{bi}$ by  $\tilde \gamma_{bi }=(s \times id )_\sharp \pi^2_\sharp \gamma_{bi}$  we obtain a new $\gamma$ 
admissible for (\ref{dualplan}) and with lower cost.
\end{proof}

For later use we set:
$$d^+_g(y):=\min_{x\in\partial \Omega} \{g(x)+\vert x-y\vert\}=g(z)+\vert z-y\vert\quad \forall z\in p^+_g(y).$$
The following Lemma will also be needed:
\begin{lemma}\label{nupos}
Let $(\gamma,\nu)$ a couple of solutions of (\ref{dualplan}) such that $\gamma_{bb}=0$. If (\ref{max}) admits a solution such that $u\ge 0$ in $\Omega$ then:
$$\nu\mbox{ is a non-negative measure and }\gamma=\gamma_{ii}+\gamma_{bi}$$ with marginals 
$\rho^+= \pi^2_\sharp \gamma$, $\rho^-=  \pi^1_\sharp \gamma_{ii}$ and $\nu=\pi^1_\sharp \gamma_{bi}.$
\end{lemma}
\begin{proof}
Using Lemma \ref{PrimalDualClassic} we get for  $\gamma_{ib}$-a.e.$(x,y)$ the equality $-u(x)=d_\Omega(x,y)$.
If $u\ge0$ this imply that $\gamma_{ib}=0$ and $\nu^-=0$.
\end{proof}
\section{Some general results about the PDE}
The following result shows that, if it exists a solution to (\ref{weakpb}), this solution is non-decreasing in time and unique.
\begin{proposition}\label{crescendo}
Suppose that $f^1\ge f^2$ and $u_0^1\ge u_0^2$. Assume that for $i=1,2$,  $(u^i,\mu_i,\nu_i)$ are solutions of (PDE) satisfying (\ref{regularity}), (C), (B1) with $f=f^i$, initial and boundary conditions:
$$u^i(x)=0\quad{(\nu^i)}^-\mbox{-a.e.}\!\ x,\quad u^i(x)=g(x)\quad{(\nu^i)}^+\mbox{-a.e.}\!\ x,$$
$$u^i(x,0)=u_0^i\quad \forall x\in \clomega.$$
Then we have $u^1\ge u^2.$ 

\end{proposition}
It is a classical result when $f$ is absolutely continuous, the following proof is adapted from \cite{CaCaSi2009CPDE}. It requires the following lemma.
\begin{lemma}\label{derivDist}
Let $v:[0,T[\times{\clomega}\rightarrow \R$ be such that
 $v\in L^1(0,T,W^{1,\infty}(\Omega))$ and\\  $\partial_t v\in L^1(0,T,\M(\Omega)).$
Then we have:
$$\displaystyle{{d \over d t}\left(\int_\Omega \vert v(x,t)\vert^2\ dx\right)=
2 \langle {\partial_t}v(x,t),v(x,t)\rangle_{\M(\Omega),\C_b(\Omega) } }\mbox{ in }\mathcal{D}'(]0,T[).$$
\end{lemma}

\begin{proof}
This is easily seen since applying Theorem 2.3.1 of \cite{Dro2001Prep}, it exists $\Phi_n\in \D(\Omega\times ]0,T[)$ such that:
$$\int_0^T \Vert \Phi_n(\cdot,t)-v(\cdot,t)\Vert_\infty dt\to 0,\quad \int_0^T \vert \partial_t \Phi_n(\cdot,t)-\partial_t v(\cdot,t) \vert(\Omega)\ dt\to 0.$$
\end{proof}
\begin{proof} (of Proposition \ref{crescendo})
By Theorem \ref{DualEXT}, we have for a.e. $t\in]0,T[$:
$$u^i(\cdot, t)\in argmax\{\langle f_t^i-{\partial_t u^i}(\cdot, t), v \rangle_{\M(\Omega),\C_b(\Omega)}\ : \ v\in \Lip_1 (\Omega),\ 0\le v\le g\mbox{ on }\partial \Omega \}.$$ 
Let $u^+(x,t)=\max\{u^1(x,t),u^2(x,t)\}$ and $u^{-}(x,t)=\min\{u^1(x,t),u^2(x,t)\}$. Using the
optimality of $u^1$ leads 
$\langle f^1-{\partial_t u^1}, u^+-u^1\rangle\le 0$ a.e. $t$
and as $f^1\ge f^2$ and $u^+\ge u^1$, we have:
\begin{equation}\label{optIn}
\langle f^2-{\partial_t u^1}, u^+-u^1\rangle\le 0\mbox{ a.e. }t.
\end{equation}
On the other side, we have:
$u^{-}-u^2=(u^1-u^2){\mathbf 1}_{\{u^1<u^2\}}=(u^1-u^+){\mathbf 1}_{\{u^1<u^2\}}=u^1-u^+$
and ${\partial_t u^2}{\mathbf 1}_{\{u^1<u^2\}}={\partial_t u^+}{\mathbf 1}_{\{u^1<u^2\}}$ a.e. $t$.
These equalities imply 
\begin{equation*}
\langle f^2-{\partial_t u^2}, u^{-}-u^2\rangle=\langle f^2-{\partial_t u^+}, u^{1}-u^+\rangle\le 0\mbox{ a.e. }t.
\end{equation*}
Combining this last inequality with (\ref{optIn}) and using Lemma \ref{derivDist}, we get:
\begin{eqnarray*}
{d\over dt}{1\over 2} \Vert u^+-u^1 \Vert_{L^2}^2&=&\langle {\partial_t u^+ }- {\partial_t u^1 }, u^{+}-u^1\rangle\\
&=&\langle -f^2+{\partial_t u^+},u^+-u^{-}\rangle+\langle f^2- {\partial_t u^1}, u^{+}-u^1\rangle\le 0\mbox{ a.e. }t.
\end{eqnarray*}
So $\Vert u^+-u^1 \Vert_{L^2}^2$ is constant in time and as at time $t=0$, it is zero, we get $u^1\ge u^2$ at any time.
\end{proof}
We are now able to prove Theorem \ref{theo1}.
As in \cite{prigozhin1996variational}, we introduce the following functional on $\C_b(\Omega)$:
$$I_\infty(v)=\left\{
\begin{array}{c l}
0 &\mbox{ if } v\in \Lip_1(\Omega) \mbox{ and } 0\le v\le g \mbox{ on } \partial \Omega,\\
+\infty & \mbox{ otherwise.}
\end{array}
\right.
$$
For any $\rho\in \M(\Omega)$, denoting $\partial I_\infty$ the subdifferential of $I_\infty$, by definition:
$$\rho \in \partial I_\infty(v) \Leftrightarrow v\in argmax \{ \langle \rho, w\rangle:\ w\in \Lip_1(\Omega) \mbox{ and } 0\le w\le g \mbox{ on } \partial \Omega\}. $$
For the record we write the following optimization problems which corresponds to (\ref{max}), (\ref{mincurrent}) and (\ref{dualplan}) with $\rho=f_t-\partial_tu$. We will make an intense use of these new notations in the present and following sections. Recall that the equality $\min(\ref{bidualT})=\max(\ref{Kantor})=\min(\ref{EqMonge})$ holds with:

\begin{equation}\label{Kantor}
\max\{\langle f_t-\partial_t u(\cdot, t),v\rangle:\ v\in \Lip_1(\Omega),\ 0\le v(x)\le g(x)\mbox{ on }\partial \Omega\},
\end{equation}
\begin{equation}\label{bidualT}
\min_{\sigma\in \M(\Omega,\R^d),\nu\in \M(\partial \Omega)} \left\{ \int d|\sigma|+\int_{\partial\Omega} g(x)d\nu^+(x):\ -{\rm div} \sigma =f_t-\partial_t u(\cdot,t)-\nu\mbox{ in }{\R^d} \right\},
\end{equation}
\begin{multline}\label{EqMonge}
\min_{\gamma \in \M^+(\clomega \times \clomega),\nu\in\M(\partial \Omega)} \left\{ \int_{\clomega \times \clomega}  \vert x-y\vert d\gamma(x,y)+\int_{\partial\Omega} g(x)d\nu^+(x): \right. \\
 \left. \pi^1_\sharp \gamma= {\partial_t u}+\nu^+, \  \pi^2_\sharp \gamma=f_t+\nu^- \right\} .
\end{multline}
We have the following result:
\begin{thm}\label{Prig}
Let $u\in L^\infty(]0,T[,W^{1,\infty}(\Omega))$ be such that $\partial_t u\in L^\infty(]0,T[,\M(\Omega))$ and 
$$ \vert Du(x,t)\vert\le 1 \mbox{ a.e. }(x,t),\quad  0\le u(x,t)\le g(x) \mbox{ on } \partial \Omega\times]0,T[.$$
\begin{itemize}
\item[1)]
\begin{itemize}
\item[(i)] If $(u,\mu,\nu)$ is a solution of (\ref{weakpb}) satisfying (\ref{regularity}) then: $$f_t-\partial_tu(\cdot, t)\in \partial I_\infty(u(\cdot,t))\mbox{ a.e.}\!\ t\in ]0,T[.$$
\item[(ii)] Assume $u$ is non-negative, $u(\cdot,0)=0$, and $f_t-\partial_tu(\cdot, t)\in \partial I_\infty(u(\cdot,t))$ $\mbox{ a.e. }\!\ t\in ]0,T[$.  Then  it exists $\mu \in  L^\infty(]0,T[,\M^+(\Omega))$ and $\nu\in  L^\infty(]0,T[,\M^+(\partial\Omega))$ such that
 $(u,\mu,\nu)$ is a solution of (\ref{weakpb}).
 
\end{itemize}
\item[2)] Assume the conditions above are satisfied.
\begin{itemize} 
\item[$\bullet$] (Uniqueness of $u$ and $\mu$)  The function $u$ is unique. Moreover, to each $\nu$ corresponds a unique $\mu$, and taking any $\gamma_t$ such that $(\gamma_t,\nu_t)$ is a solution of (\ref{EqMonge}) a.e. $t\in ]0,T[$, the following formula holds: 
 $$\langle\mu_t,\varphi \rangle_{\M(\Omega),\C_b(\Omega)}:=\int_{\Omega^2}\int_0^1 \varphi((1-s)x+sy )\vert y-x\vert\ ds d\gamma_t(x,y)\quad \forall \varphi\in \C_b(\Omega).$$
 \item[$\bullet$] For a.e. $t\in ]0,T[$, $u(\cdot, t)$ is space differentiable $\mu_t$-almost everywhere and: $D_{\mu_t} u(x,t)=Du(x,t)\quad \mu_t\mbox{-a.e. }x.$
 \end{itemize}
 \end{itemize}
\end{thm}
\begin{rmk}\label{PrigRmk}
a) In the original article \cite{prigozhin1996variational}, in case $g=0$ and $f\in L^\infty(]0,T[,L^\infty(\Omega))'$, L. Prigozhin, in the first place, proved (a similar result to)  point 1). In the case he considered, $\mu$ is expected to be in $ L^\infty(]0,T[,W^{1,\infty}(\Omega))'$ and (PDE) has to be understood in the following sense:
\begin{equation}\label{PDEPrig}
\langle \partial_t u- f, \varphi\rangle+ \langle \mu, Du\cdot D\varphi\rangle=0\quad \forall \varphi\in L^\infty(]0,T[,L^\infty(\Omega)).
\end{equation}
 In our case, (\ref{PDEPrig}) can be recovered by extending $\mu$ in   $ L^\infty(]0,T[,L^\infty(\Omega))'$ using the Hahn-Banach Theorem (cf Theorem \ref{Prig}, point 2)).\\
b) The proof of 1)(ii) below says actually more than required. Assume  that $f_t-\partial_tu(\cdot, t)\in \partial I_\infty(u(\cdot,t))$ for a.e.$\!\ t\in ]0,T[$ and take, for a.e. $t$, $(\sigma_t,\nu_t)$  {\bf any} solution of (\ref{bidualT}) with $\nu_t\ge 0$. Then, under the assumption of the theorem, setting $\mu_t=\vert \sigma_t\vert$, we have that $(u,\mu,\nu)$ satisfies (PDE) and $$\vert D_\mut u(x,t)\vert =1\  \mu\mbox{-a.e.}\!\ (x,t),\quad u(x,t)=g(x)\  \nu\mbox{-a.e.}\!\ (x,t).$$
 Moeover we have: $\sigma_t= D_\mut u(\cdot, t) \mu_t$. 
\end{rmk}

\begin{proof}
\ 1)(i) is an immediate consequence of Theorem \ref{DualEXT}  (since $\nu_t\ge 0$).\\
\ 1) (ii) With these assumptions, $u(\cdot, t)$ is optimal for (\ref{Kantor}) for a.e. $t$. Take $(\sigma_t,\nu_t)$ a solution of (\ref{bidualT}) with $\nu$ non-negative: this is possible by lemma \ref{nu-nu} and \ref{nupos}.  By Theorem \ref{DualEXT} (ii), setting $\mu_t=\vert \sigma_t\vert$ we get the result.\\
\ 2) The uniqueness property of $u$ comes from Proposition \ref{crescendo} while the properties of $\mu$ are contained in Lemma \ref{nu-nu} as $(\sigma_t:=D_{\mu_t} u\mu_t,\nu_t)$ is a solution of (\ref{bidualT}) by Theorem \ref{DualEXT}.
For the last property, see Lemma \ref{nu-nu}.
\end{proof}


\section{The case of a finite number of sources}\label{Discret}

In the spirit of \cite{aronsson1972mathematical} and \cite{AroEvaWu1996JDE}, we are now looking at (\ref{weakpb}) when  $f_t=\sum_{j=1}^k c_{j}\delta_{y_j}$. As it is constant in time we will often write $f$ for $f_t$.

In the next two lemmas we develop an heuristic of the shape of solutions for this special $f$. Starting from this idea we then show existence of a solution in Theorem \ref{solDisc}.  

Assume for a while that a solution $u\in L^\infty(0,T,W^{1,\infty}(\Omega))$ of  (\ref{weakpb}) is known. By the previous section, we know $u(\cdot,t)$ is a solution of (\ref{Kantor}) for a.e. time, it's non-negative and non-decreasing in time. Moreover take $(\gamma_t,\nu_t)$  a solution of (\ref{EqMonge})  for a.e. $t\in[0,T]$ with the following decomposition given by Lemma \ref{nupos}:
$$\gamma_t=\gamma_{ii,t}+\gamma_{bi,t}\mbox{ and }\nu_t=\pi^1_\sharp\gamma_{bi,t}\ge 0.$$ The Lemmas \ref{miniU} and \ref{derivt} will give us some clues to guess the shape of $u$:

\begin{lemma}\label{miniU}
 Let us set: 
$$r_j(\cdot):=u(y_j,\cdot),\quad {\underline u}(x,t):=\max_j\{r_j(t)-\vert x-y_j\vert,0\}.$$ 
Then ${\underline u}(\cdot,t)$ is in $\Lip_1(\Omega)$ for all  $t\in]0,T[$, $r_j\in {\mathcal C}([0,T],\R^+)$ and they satisfy:
\begin{itemize}
\item[(i)] $u(x,t)\ge {\underline u}(x,t)$ for all $(x,t)$ with equality for $x=y_i$ and any $x\in\spt(\partial_t u(\cdot,t))\cup \spt(\nu_t)$, 
\item[(ii)] $r_j(t)\in[0,d^+_g(y_j)]$ with  $r_j(t)=d^+_g(y_j)=g(x)+\vert x-y_j\vert\quad \gamma_{bi,t}$-a.e. $(x,y_j)$ for a.e. $t\in [0,T]$. 
\end{itemize}
Moreover $0\le\underline{u}(x,t) \le g(x)$ on $\partial \Omega$ for all $t$ and ${\underline u}(\cdot, t)$ is also optimal for (\ref{Kantor}).
\end{lemma}

\begin{proof}
\begin{itemize}
\item[(i)] 
The inequality comes from the Lipschitz property of $u$ and the definition of $r_j$. The equality at every $y_j$ follows from the definition $r_j(t)=u(y_j,t)\ge 0$. Then by Lemma \ref{PrimalDualClassic}: 
\begin{equation}\label{PD}
u(y_j,t)-u(x,t)=\vert x-y_j\vert \quad \gamma_t\mbox{-a.e.}(x,y_j).
\end{equation}
 This gives the equality on $\spt(\partial_t u(\cdot,t))\cup \spt(\nu_t)$ since $\pi^1_\sharp \gamma=\partial_t u(\cdot,t)+ \nu_t$.
\item[(ii)] Consider $x\in p_g^+(y_j)$, by the boundary condition (B1) and the Lipschitz property of $u$:
$$r_j(t):=u(y_j,t)\le u(x,t)+\vert x-y_j\vert\le g(x)+\vert x-y_j\vert= d^+_g(y_j)\mbox{ for all }j=1,...,k.$$
This implies $\underline{u}(x,t) \le g(x)$ on $\partial \Omega$.
Moreover, by Proposition \ref{projection} combined with (\ref{PD}) every inequality above becomes an equality for $\gamma_{bi,t}$-a.e $(x,y_j)$.
\end{itemize}
As ${\underline u}(\cdot, t)$ is admissible for (\ref{Kantor}) and equals $u(\cdot, t)$ on $\spt(f)\cup \spt(\partial_t u(\cdot,t))\cup (\spt(\nu_t))$, it is optimal.
\end{proof}
 
For $j=1,\dots, k$, let us introduce the following subset of $\clomega$:
\begin{equation}\label{defaj}
A_j(t):=\{ x\in \clomega:\ r_j(t)-\vert x-y_j\vert=\max_n \{ r_n(t)-\vert x-y_n\vert, 0\} \}, 
\end{equation}
\begin{equation}\label{u-shape}
\mbox{ so that }{\underline u}(x,t)=\sum_{j=1}^k(r_j(t)-\vert x-y_j\vert){\mathbf 1}_{A_j (t)} (x).
\end{equation}

\begin{lem}\label{derivt} Assume  for all $j=1,...k$, $r_j$ is derivable for almost every $t\in]0,T[$. Then
for all $t>0$ and a.e. $x$: 
$$ \partial_t \underline{u}(x,t)= \sum_{j=1}^k \dot{r}_j (t) {\mathbf 1}_{A_j (t)} (x)\mbox{ a.e. }(x,t).$$
\end{lem}

\begin{proof}
We have ${\rm Int}(A_j(t))=\{x\in \Omega:\    r_j(t)-\vert x-y_j\vert> \max_{m\not =j}\{r_m(t)-\vert x-y_m\vert,0\}\}$ and the boundary of $A_j(t)$ is negligible.
Indeed the sets $A_m(t)\cap A_j(t)=\{x\in \clomega:\  \vert x-y_m\vert-\vert x-y_j\vert=r_m(t)-r_j(t)\}$ and $\{x\in \clomega:\ \vert x-y_j\vert=r_j(t)\}$ are negligible as $d\ge 2$ and $\Omega$ is convex.\\  
Let $x$ in the interior of $A_j(t)$, then $r_j(t)-\vert x-y_j\vert > \max_{m\not =j}\{r_m(t)-\vert x-y_m\vert,0\}$. As $r_j\in\C([0,T])$, this inequality remains true at time  $t+h$ with $\vert h\vert$ small enough, as a consequence if $r_j$ is derivable on $t$:
$\lim_{h\to 0}{\underline{u}(x,t+h)-\underline{u}(x,t)\over h}=\lim_{h\to 0}{r_j(t+h)-r_j(t)\over h}=\dot r_j(t).$
We then get $\partial_t \underline{u}(x,t)={\dot r}_j(t)$ for all $x\in {\rm Int}(A_ j(t))$ for almost every $t$.\\
\end{proof}

In the sequel, we are going to see that, with the appropriate choice of  $r_j$, the application $u=\underline u$ is the unique solution of (\ref{weakpb}).\\
By Lemma \ref{crescendo}, if $u$ is a solution of (\ref{weakpb}), the function $r_j:=u(y_j,\cdot)$ are non-decreasing, moreover by the initial condition (I), $r_j(0)=0$. Take $\gamma_t$ optimal for (\ref{EqMonge}). We deduce that, for small times $r_j(t)<d^+_g(y_j)$ and by Lemma \ref{miniU}, for such $t$, $\gamma_{bi,t}$ and $\nu_t$ are $0$.
So that, if $u=\underline u$, then, for small $t$, the map $\gamma_t$ has marginals $\sum_{j=1}^k \dot{r}_j (t) {\mathbf 1}_{A_j (t)} (x)\mbox{ a.e. }(x,t)$ and $f=\sum_{j=1}^k c_j\delta_{y_j}$ by (\ref{PD}) and (\ref{u-shape}), this implies:
\begin{equation}\label{massBal}
\dot{r}_j (t) \vert  A_j (t)\vert=c_j \quad \forall j=1,...,k,
\end{equation}
$$\gamma_t(x,y)=\sum_{j=1}^k \dot{r}_j (t) {\mathbf 1}_{A_j (t)} (x)\otimes \delta_{y_j}(y).$$ 
Finally, every time (\ref{massBal}) holds, $r_j$ remains strictly non-decreasing and by Lemma \ref{miniU}, when $r_j$ reaches $d^+_g(y_j)$, it cannot increase anymore.\\ All this remarks leads us to look
at the ODE (\ref{ODE}) \footnote{ This ODE first appeared in \cite{aronsson1972mathematical}, see also \cite{AroEvaWu1996JDE}.}.

\begin{lem}\label{lemODE} There  exist times $(t_1,...,t_k)\in \R_+^k$ and functions  $r_j\in \C^1(0,t_j)\cap \C^0 ([0,+\infty))$ which satisfy:
\begin{equation}\label{ODE}
\left\{
\begin{array}{l}
\dot r_j(t)=\frac{c_j}{\vert A_j(t)\vert} \quad \forall t\in ]0,t_j[,\\
r_j(0)=0,  \\
r_j(t)=d^+_g(y_j)=\min_{x\in \partial\Omega}\{g(x)+\vert x-y_j\vert\}\quad \forall t>t_j\\
\end{array}
\right.
\end{equation}
with $A_j(t)$ defined above in (\ref{defaj}). Moreover $y_j\in A_j$ for all $t$ and if $t<t_j$ then $0< | A_j (t) |$.
\end{lem}

\begin{proof} {\bf Step 1}: Let $m_1= \min_{i\not=j} |y_i-y_j|$, $m_2= \min_i d(y_i, \partial \Omega)$ and $m=\min\{m_1, m_2\}$, also consider $c=\max_i c_i$.
For small times  i.e. for 
$$ t \leq \frac{\omega_d }{ (d+1)c}(\frac{m}{2})^{d+1}:=t_0$$
the functions
$r_ i(t)=  (\frac{c_i}{\omega_d} (d+1) t)^{1/d+1}$
are solutions of the ODE above with the correct initial data.\\
\noindent {\bf Step 2}: Starting from $t_0$ the sets where $r_i(t)-|x-y_i| >0$  may start to intersect  or may touch the boundary of $\Omega$ and for the existence of solutions we appeal to 
a standard existence theorem. This requires that we prove that the functions
$$ (r_1, \dots r_k) \mapsto \frac{1}{|A_i(r_1, \dots r_k)|}$$
are continuous \footnote{This holds true if $d>1$.}. 
First we prove that,  for every $i\in\{1,\dots,k\}$,
$ (r_1, \dots r_k) \mapsto |A_i(r_1, \dots r_k)|$
is continuous  on $\R^+\times \dots \times \R^+$.

Let $\varepsilon=(\varepsilon_1, \dots, \varepsilon_k) $. Define $\overline{\varepsilon}= \max_j  \{\varepsilon_i, \varepsilon_i-\varepsilon_j\}$ and 
$\underline{\varepsilon}= \min_j  \{\varepsilon_i, \varepsilon_i-\varepsilon_j\}$.

\begin{eqnarray*}
& & A_i(r_1+ \e_1 ,\dots, r_k+ \e_k) =\\
&=& \{x\in \clomega :\  0 \leq r_i+\e_i -|x-y_i|, \ r_j+\e_j -|x-y_j| \leq   r_i+\e_i-|x-y_i| \ \forall \ j \}\\
&\subset & \{x\in \clomega :\  0 \leq r_i+\overline{\varepsilon}-|x-y_i|, \ r_j -|x-y_j| \leq   r_i+\overline{\varepsilon}-|x-y_i| \ \forall \ j \}
 \\
&= & A_i(r_1,\dots,r_i+\overline{\varepsilon}, \dots,  r_k)
\end{eqnarray*}
and similarly 
$ A_i(r_1,\dots,r_i+\underline{\varepsilon}, \dots,  r_k)\subset A_i(r_1+ \e_1 ,\dots, r_k+ \e_k) .$ 

If  $\varepsilon _n \searrow 0 $ then 
$A_i(r_1,\dots,r_i+\varepsilon_{n+1}, \dots,  r_k) \subset A_i(r_1,\dots,r_i+\varepsilon_n, \dots,  r_k),$
and 
$ A_i(r_1,\dots, r_k) =\cap_n    A_i(r_1,\dots,r_i+\varepsilon_n, \dots,  r_k),$ 
it follows that
$$ |A_i(r_1,\dots, r_k) |= \lim_{n \to 0} | A_i(r_1,\dots,r_i+\varepsilon_n, \dots,  r_k)|.$$
If  $\varepsilon _n \nearrow 0 $ 
then  
$A_i(r_1,\dots,r_i+\varepsilon_{n}, \dots,  r_k) \subset A_i(r_1,\dots,r_i+\varepsilon_{n+1}, \dots,  r_k),$
and 
\begin{multline*}
 A_i(r_1,\dots, r_k) \setminus ( \{x\in \clomega :\  0 = r_i -|x-y_i|\} \cup \cup_j \{x\in \clomega :\  r_j -|x-y_j| =  r_i-|x-y_i|\}  ) =\\ \cup_n    
A_i(r_1,\dots,r_i+\varepsilon_n, \dots,  r_k),
\end{multline*}
and since\footnote{and here we are using both the convexity of $\Omega$ and the dimension $2\leq d$} 
$$ |\{x\in \clomega :\  0 = r_i -|x-y_i|\}|=0\ \mbox{and}\ |\{x\in \clomega :\  r_j -|x-y_j| =  r_i-|x-y_i| \} |=0$$
it follows that
$$ |A_i(r_1,\dots, r_k) |= \lim_{n \to 0} | A_i(r_1,\dots,r_i+\varepsilon_n, \dots,  r_k)|.$$

\noindent {\bf Step 3}: To get the existence of a solution of the ODE, it only remains to show that the measures of $A_i(t)$ do not tend to $0$ when $r_i(t)<d_g^+(y_i)$. By the triangular inequality, if  $A_i \neq \emptyset $ then $y_i \in A_i$. 
If  $\lim_{t \to \underline{t}} |A_i(t)|=0$ then there exists at least one $j \neq i$ such that $\lim_{t \to \underline{t}} r_j(t)-r_i(t) = |y_i-y_j|$.  
Indeed, if for all $j \neq i$  we have $ 0 > L_j= \lim_{t \to \underline{t}} r_j(t)-r_i(t)- |y_i-y_j|$ then there exists $0 \leq l$ such that 
$ B(y_i, l) \subset A_i(s)$ for $s$ close enough to $\underline{t}$ which prevents the measure of $A_i$ from going to $0$.

Then we proceed by contradiction and assume it exist $\underline t$ and $i_0$ such that $\lim_{t \to \underline{t}} |A_{i_0}(t)|=0$ and $r_{i_0}(t)<d_g^+(y_{i_0})$.
Let $I(\bar t):=\{i:\ r_i(\bar t)< d_g^+(y_i)\}$. Since $t\mapsto r_i(t)$ is non-decreasing
$$\sum_{i\in I(\bar t)} |A_i(t)| \ge\sum_{i\in I(\bar t)} |A_i(t_0)| >0$$ so there exist at least two indices $(m,j)\in I(\bar t)\times I(\bar t)$ such that 
\begin{enumerate}
\item $\lim_{t \to \underline{t}} r_m(t)-r_j(t) = |y_j-y_m|,$
\item $\lim_{t \to \underline{t}} |A_j(t)|=0,$
\item $\lim_{t \to \underline{t}} |A_m(t)|=L>0.$
\end{enumerate}
But this is impossible since, taking a derivative (as $j\in I(\bar t)$),  we also deduce $\lim_{t \to \underline{t}} r_m(t)-r_j(t) = -\infty.$
Which concludes the proof.

\noindent {\bf Step 4}: Now we observe that
$$ \frac{c_i}{|A_i|} \geq \frac{c_i}{|\Omega|}$$
and then $r_i$ reaches the value $d_ g^+ (y_i)$ in finite time $t_i$.

\end{proof}
We are now able to give the solutions of (\ref{weakpb}) and the associated optimization problems:
\begin{itemize}
\item[$\bullet$] Set $u(x,t):=\sum_{j=1}^k(r_j(t)-\vert x-y_j\vert){\mathbf 1}_{A_j (t)} (x)$ with $r_j$ and $A_j$ as in the previous lemma;
\item[$\bullet$] Take $I_1(t)=\{j:t<t_j\}$ and $I_2(t)=\{j:t>t_j\}$ two families of indices and set for any $t\not\in\{t_j:\ j=1,...,k\}$:
$$\nu_t:=\sum_{j\in I_2(t)} \nu_j\mbox { with $\nu_j\in \M_b^+(\partial \Omega)$ such that: }\nu_j(\partial \Omega)=c_j,\ \spt(\nu_j)\subset p_g^+(y_j),$$
$$\gamma_t=\gamma_{ii,t}+\gamma_{bi,t},\quad \gamma_{ii,t}(x,y)=\sum_{j\in I_1(t)} {c_j\over \vert A_j\vert} {\mathbf 1}_{A_j (t)} (x)\otimes \delta_{y_j}(y),\quad \gamma_{bi,t}=\sum_{j\in I_2(t)} \nu_j\otimes \delta_{y_j};$$
 \item[$\bullet$] For a.e. $t$, define $\mu_t\in \M(\Omega)$ as: 
 $$\int_{\Omega} \varphi(x)d\mu_t(x) =\int_{\Omega^2}\int_0^1 \varphi((1-s)x+sy)\times \vert y-x\vert \ ds d\gamma_t(x,y),\ \forall \varphi\in \C_b(\Omega).$$

\end{itemize}

The following result holds:

\begin{thm}\label{solDisc} 
With the above definitions, (\ref{regularity}) holds, 
moreover:
\begin{enumerate}
\item[1)] The triplet $(u,\mu,\nu)$ is a solution of (\ref{weakpb}), with uniqueness on $u$,
\item[2)] For a.e. $t\in[0,T]$, $u(\cdot,t)$ is a solution of (\ref{Kantor}) and $(\gamma_t,\nu_t)$ is a solution of (\ref{EqMonge}),
\item[3)] The couple $({D}_{\mu_t} u(\cdot,t)\mu_t,\nu_t)$ is a solution of (\ref{bidualT})  a.e. $t\in[0,T]$.
\end{enumerate}
\end{thm}

\begin{proof}
{\bf Step 1:} We show first that (\ref{regularity}) is satisfied. Indeed for a.e. $t\in]0,T[$ by (ODE):
$$0\le \int_{\Omega} \partial_t u(\cdot, t)=\sum_{j\in I_1} \int_{A_j} {\dot r}_j(t)\ dx=\sum_{j\in I_1} c_j\le \int_{\Omega} df(x),$$
by definition of $\nu$:
$$\nu_t(\partial \Omega)=\sum_{j\in I_2} c_j\le \int_{\Omega}  df(x),$$
and finally:
\begin{eqnarray*}
&&\int_{\Omega} d\mu_t(x)=\sup_{\Vert\varphi\Vert_\infty=1}\left\vert\int_{\Omega^2}\int_0^1 \varphi((1-s)x+sy)\times \vert y-x\vert\ ds d\gamma(x,y)\right\vert\\
&&\quad \le\  {\rm diam}(\Omega)\times\gamma(\Omega^2)\ =\ {\rm diam}(\Omega)\times\int_{\Omega} df(x).
\end{eqnarray*}
\noindent{\bf Step 2:}  Note that, by Lemma \ref{lemODE}, as $y_j\in A_j(t)$, we have $u(y_j,t)=r_j(t)$ for all time $t\in]0,T[$. Let us now prove 2). Indeed, as  for any $j\in I_1$, $c_j=\dot{r}_j \vert A_j\vert$, we have for a.e. $t\in ]0,T[$:
\begin{eqnarray*}
&&\langle \sum_{j=1}^k c_j\delta_{y_j}-\partial_t u(\cdot, t), u(\cdot,t) \rangle_{\M(\Omega),\C_b(\Omega)}= 
\sum_{j=1}^k c_j u(y_j, t)-\sum_{j\in I_1} \int_{A_j} u(x,t) {\dot r}_j(t)\ dx\\
&&=\sum_{j\in I_1} \int_{A_j} [r_j(t)-u(x,t)]\  {\dot r}_j(t)\ dx +\sum_{j\in I_2} c_jr_j(t)
= \sum_{j\in I_1} \int_{A_j} \vert x-y_j\vert\ {\dot r}_j(t)\ dx +\sum_{j\in I_2} c_jd^+_g(y_j)\\
&&=  \int_{\Omega^2} \vert x-y\vert\ d\gamma_{ii,t}(x,y) +\sum_{j\in I_2} \int_{p_g^+(y_j)} \vert x-y_j\vert+g(x)\ d\nu_j(x)\\
&&=  \int_{\Omega^2} \vert x-y\vert\ d\gamma_{ii,t}(x,y) + \int_{\clomega^2} \vert x-y\vert \ d\gamma_{bi,t}(x,y)+ \int_{\partial \Omega} g(x)\ d\nu_t(x)\\
&&=  \int_{\clomega^2} \vert x-y\vert\ d\gamma_{t}(x,y) +\int_{\partial \Omega} g(x)\ d\nu_t(x).
\end{eqnarray*}
As we know $u(\cdot, t)$ is admissible by Lemma \ref{miniU} so, by duality (see Remark \ref{PrimD}), it is optimal and also is $(\gamma_t,\nu_t)$.\\
\noindent {\bf Step 3:}  Lemma \ref{nu-nu} gives a measure $\sigma_t \in \M_b(\Omega, \R^d)$ such that $(\sigma_t,\nu_t)$ is an optimal solution of (\ref{bidualT}) for a.e. $t\in]0,T[$: 
$$\langle\sigma,\phi \rangle_{\M(\Omega)^d,\C_b(\Omega)^d}:=\int_{\Omega^2}\int_0^1 \phi((1-s)x+sy)\cdot (y-x)\ ds d\gamma_t(x,y).$$
 Then by Theorem \ref{DualEXT} $\sigma_t= {D}_{\vert \sigma_t\vert} u(\cdot,t) \vert \sigma_t\vert$ so 1) and 3) are proved (as $\mu_t=\vert \sigma_t\vert$).\\
\end{proof}

Note for later use the following estimations:
\begin{proposition}\label{EstimDis}
The following estimates hold for a.e. $t\in ]0,T[$:
\begin{eqnarray*}
\Vert u(\cdot,t)\Vert_\infty&\le& \Vert g\Vert_\infty +  {\rm diam}(\Omega),\qquad \Vert Du(\cdot,t)\Vert_\infty\le1,\\
 \nu_t(\partial \Omega)&\le& \int_{\Omega} df(x),\qquad \int_{\Omega} \partial_t u(\cdot, t) dx\le \int_{\Omega} df(x),\\
&&\int_{\Omega} d\mu_t(x)\le {\rm diam}(\Omega)\times\int_{\Omega} df(x).
\end{eqnarray*}
Moreover, we have:
$$\Vert \partial_t u\Vert_{L^2(\Omega\times ]0,T[)}\le \Vert u(\cdot,t)\Vert_\infty\times \int_{\Omega} df(x).$$
\end{proposition}
\begin{proof}
The only remaining points are the estimates on $\Vert u(\cdot,t)\Vert_\infty$ and $\Vert \partial_t u\Vert_{L^2(\Omega\times ]0,T[)}$. By the Lipschitz property of $u$ and condition (B1) taking $x\in\clomega$ and $y\in \partial \Omega$ leads: $ 0 \le u(x,t) \le g(y)+\vert x-y\vert$ for all $t\in]0,T[$. This gives the first inequality. Let us now show the last one:
\begin{eqnarray*}
\Vert \partial_t u\Vert_{L^2(\Omega\times ]0,T[)}&=&\int_0^T \int_{\Omega} \sum_{j=1}^k ({\dot r}_j(t))^2 \1_{A_j(t)}(x)\ dt dx
\  =\     \int_0^T \sum_{j=1}^k ({\dot r}_j(t))^2\vert A_j(t)\vert dt\\&=&\int_0^T \sum_{j=1}^k c_j{\dot r}_j(t)dt\  =\   \sum_{j=1}^k c_ju(y_j,T)\   \le\    \Vert u(\cdot,t)\Vert_\infty\int_{\Omega} f(x) dx.\\
\end{eqnarray*}
\end{proof}

\section{A more general case}
Let $f_t=f\in \M^+(\Omega)$, constant in time, we now aim to show that (\ref{weakpb}) admits a solution (which will be unique by Proposition \ref{crescendo}). \\
We approximate $f$ by a sequence $(f^n)_n$ in $\M^+(\Omega$) such that:
$$f^n=\sum_{i=1}^n c_i^n\delta_{y_i^n}\mbox { with } y_i^n\in \Omega,\quad f^n(\Omega)\le f(\Omega),\quad f^n\ws f\mbox{ in }\Omega.$$
By the last section, it exists $(u^n,\mu^n,\nu^n)_n$ in $L^\infty(0,T;W^{1,\infty}(\Omega))\times L^\infty(0,T;\M(\Omega))\times L^\infty(0,T;\M^+(\partial\Omega))$ satisfying Theorem \ref{solDisc} and Proposition \ref{EstimDis}.

\begin{proposition}(Convergence of $u^n$, $\partial_t u^n$, $\nu^n$)\label{cvUn}
There exist $u\in L^\infty(0,T;W^{1,\infty}(\Omega))$ and $\nu\in  L^\infty(0,T;\M(\partial\Omega))$ such that:
\begin{enumerate}
\item[i)] $u(\cdot, t)\in{\Lip}_1(\Omega)$, $\partial_t u\in L^\infty(0,T;\M^+(\partial\Omega))\cap L^2(\Omega\times ]0,T[)$ and $u$ satisfies the conditions $(B_1)$ and $(I)$ of (\ref{weakpb});
\item[ii)] up to a subsequence, $u^n$ converges to $u$ for the strong topology of $L^1(0,T,C_b(\Omega))$ and $\partial_t u^n$ converges to $\partial_t u$ for the weak star topology of  $L^\infty(0,T,\M(\Omega))$;
\item[iii)] $\nu(\cdot, t)\in \M^+(\partial \Omega)$ a.e. $t$ and $\nu^n$ converges to $\nu$ for the weak star topology of $L^\infty(0,T,\M(\partial\Omega))$ (up to a subsequence) and $(B_2)$ is satisfied.
\end{enumerate}
\end{proposition}

\begin{proof}
{\bf Step 1}: By the estimates on $(u^n)_n$,  it is a bounded sequence in $BV(\Omega\times ]0,T[)$, as a consequence (see Theorem 4 p176 of \cite{EvansGariepy1992}):
$$u^{n_k}\to u\mbox{ in }L^1(\Omega\times]0,T[)\mbox{ and for a.e. }t $$ 
for some $u\in L^1(\Omega\times]0,T[)$ and some subsequence $(u^{n_k})_k$ of $(u^n)_n$. Possibly extracting again a subsequence, we can assume: 
$$u^{n_k}(x,t)\to u(x,t)\mbox{ for almost every $(x,t)\in \Omega\times]0,T[$.}$$

{\bf Step 2}: By Ascoli-Arzela Theorem, for all $t\in [0,T[$, $(u^{n_k}(\cdot ,t))_k$ admits a cluster point $v_t\in{\Lip}_1(\Omega)$ (depending on $t$) for the uniform convergence. 
Using Step 1, we get $u(x, t)=v_t(x)$  for almost every $(x,t)\in \Omega\times]0,T[$. This shows the uniqueness of the cluster point of $(u^{n_k}(\cdot ,t))_k$ for the uniform convergence, so that:
$$u^{n_k}(\cdot,t)\to u(\cdot, t)\quad\forall t\in[0,T[, \quad\mbox {uniformly in }\clomega.$$
This gives $(B_1)$ and $(I)$. To get the convergence in  $L^1(0,T,\C_b(\Omega))$ of $(u^{n_k})_k$, we use the dominated convergence Theorem and the following bound ($u^{n_k}$ is in $\Lip_1(\Omega)$ and it is bounded by $\Vert g\Vert_\infty$ on $\partial\Omega$):
$$\sup_{t,x}\vert u^{n_k}(x,t)-u(x,t)\vert \le 2 ({\rm diam}(\Omega)+ \Vert g\Vert_\infty).$$

{\bf Step 3}:  $L^\infty(0,T;\C_b(\Omega))$ being separable, (see for instance \cite{Dro2001Prep}, Corollary 1.3.2. p 13), 
using the bound on $\partial_t  u^n$, possibly passing to a subsequence, it admits a limit $m\in L^\infty(0,T,\M(\Omega))$ for the weak star topology. It can easily be seen that $m=\partial_t u$ in the sense of distribution. In the same way, $\partial_t u\in L^2(\Omega\times ]0,T[)$ and, up to a subsequence $\partial_t u^n\rightharpoonup \partial_t u$ in $L^2$.\\

The point iii) is left to the reader. To get $(B_2)$ just note that
$$0=\lim_{n\to +\infty}\langle \nu_n,g-u_n\rangle= \langle \nu,g-u\rangle$$
up to some subsequence.
\end{proof}

\begin{rmk}\label{unifConv}
The proof shows also a different kind of convergence for $(u_{n})_n$: up to a subsequence, $u^n(\cdot,t)$ converges uniformly to $u$ for all $t$. As we will prove, $u$ is the unique solution of (\ref{weakpb}), as a consequence the complete sequence $(u^n(\cdot, t))_n$ converges to $u$ uniformly for all $t$.  
\end{rmk}

\begin{proposition}(Convergence of $\mu^n$ and $D_{\mut^n}u^n\mu^n$) \label{convMU}
It exists $\mu\in  L^\infty(0,T,\M(\Omega))$ and $\xi\in L^2_{\mu}(\Omega)^d$ such that, up to a subsequence:
$$\mu^n\ws \mu \mbox{ for the weak star topology of }L^\infty(0,T,\M(\Omega)),$$
$$D_{\mu^n_t}u^n\mu^n\ws \xi \mu  \mbox{ for the weak star topology of }\M(\Omega\times]0,T[)^d,$$
$$\liminf_{n\to+\infty} \int_0^T\int_{\Omega}d\mu^n\ge \int_0^T\int_{\Omega} \vert \xi(x,t) \vert d\mu(x,t).$$ 
\end{proposition}

\begin{proof}
The convergence of $(\mu^n)_n$ follows from its boundedness. Moreover, we have:
$$ \int_0^T\int_{\Omega} \vert D_{\mu_t^n}u^n(x,t)\vert^2 d\mu^n(x,t)=\int_0^T\int_{\Omega}d\mu^n\le T\times \int_\Omega df$$
so  Lemma 3.3 p13 of \cite{BoBuDeP2010JCA} applies and we get the rest of the proposition as:
$$ \int_0^T\int_{\Omega} \vert D_{\mu^n_t}u^n(x,t)\vert d\mu^n(x,t)= \int_0^T\int_{\Omega}d\mu^n.$$

\end{proof}

\begin{proposition}\label{limitPDE}(Passing to the limit in the PDE)
The following equality holds:
$$\xi=D_\mut u\quad \mu\mbox{-a.e.}(x,t),\quad \vert D_\mut u\vert =1\quad\mu\mbox{-a.e.}(x,t).$$
Moreover:
 $$-{\rm div}(D_\mut u\mut)=f-\partial_t u-\nu\quad\mbox { in }\R^d\times ]0,T[.$$
\end{proposition}

\begin{proof}
Passing to the limit, it is easily seen that: 
 $$-{\rm div}(\xi \mut)=f-\partial_t u-\nu\quad\mbox { in }\R^d\times ]0,T[.$$
So that $\xi(\cdot, t)\mu_t$ is admissible for (\ref{bidualT}) at almost every $t\in ]0,T[$. As, by Proposition \ref{cvUn}, $u(\cdot, t)$ is admissible for (\ref{Kantor})  at almost every $t\in ]0,T[$, by  duality we have:
\begin{equation}\label{InegDual}
\langle f-\partial_tu(\cdot, t), u(\cdot, t)\rangle_{\M( \Omega), \C_b(\Omega)}\le \sup(\ref{Kantor})=\inf(\ref{bidualT})\le \int_{\Omega}\vert \xi(\cdot,t)\vert\ d\mu_t+\int_{\partial \Omega} g(x) d\nu_t(x)\ \mbox{ a.e.}t.
\end{equation}
By Theorem \ref{solDisc},  we have:
\begin{equation*}
\int_0^T\langle f^n-\partial_tu^n(\cdot, t), u^n(\cdot, t)\rangle_{\M( \Omega), \C_b(\Omega)}\ dt= \int_0^T\int_{\Omega}d\mu^n+\int_0^T\int_{\partial \Omega} g(x) d\nu^n(x,t).
\end{equation*}
Then by Proposition \ref{convMU}:
\begin{eqnarray*}
&&\int_0^T\langle f-\partial_tu(\cdot, t), u(\cdot, t)\rangle\ dt=\lim_{n\to +\infty}\int_0^T\langle f^n-\partial_tu^n(\cdot, t), u^n(\cdot, t)\rangle\ dt\\
&&=\liminf_{n\to +\infty} \int_0^T \int_{\Omega}d\mu^n+\int_0^T\int_{\partial \Omega} g(x) d\nu^n(x,t)\ge \int_0^T \int_{\Omega}\vert \xi\vert\ d\mu+\int_0^T\int_{\partial \Omega} g(x) d\nu(x,t).
\end{eqnarray*}
This implies that (\ref{InegDual}) is an equality and $u$ is optimal for (\ref{Kantor}) and $(\xi\mu,\nu)$ is optimal for (\ref{bidualT}). Then by Theorem \ref{Prig} and Remark \ref{PrigRmk}, we get the desired result.
\end{proof}

To conclude (recall Remark \ref{unifConv}), we have proved the following result:

\begin{thm}
Let $f\in \M^+(\Omega)$, the equations (\ref{weakpb}) admit a solution $(u,\mu,\nu)$ satisfying (\ref{regularity}). Moreover:
\begin{itemize}
\item[1)] $u$ is the unique solution of (\ref{weakpb}), it is non decreasing.  For a.e.\!\ $t$, the measure $\partial_t u(\cdot, t)$ is absolutely continuous with respect to the Lebesgue measure. Actually: $\partial_t u\in L^2(\Omega \times ]0,T[).$
\item[2)] $\nu$ is non-negative, supported on $p_g^+(\spt (f))$. Moreover, to any $\nu$ corresponds a unique $\mu$ which can  be built as:
$$\langle\mu_t,\varphi\rangle=\int_{\clomega^2}\int_0^1\varphi((1-s)x+sy)\vert y-x\vert\ ds\ d\gamma_t(x,y)\quad \forall \varphi\in \C_b(\Omega).$$
where $\gamma_t$ is any solution of (\ref{EqMonge}).  
\item[3)] For a.e. $t\in ]0,T[$, $u(\cdot, t)$ is space differentiable $\mu_t$-almost everywhere and: $D_{\mu_t} u(x,t)=Du(x,t)$, $\mu_t\mbox{-a.e. }x.$\\
Moreover, for any time such that $u(\cdot, t)<g$ on $\partial \Omega$ we have $\nu_t=0$,  $\mu_t<< {\mathcal L}_{\Omega},$ and  uniqueness on $\mu_t$. 
\item[4)] Taking $f^n=\sum_{i=1}^n c_i^n\delta_{y_i^n}$ any sequence converging (weak star) to $f$ with $y_i^n\in \Omega$ and $f_n(\Omega)\le f(\Omega)$,  $u$ can be obtained as the uniform limit for all $t$ of the sequence $(u^n)_n$ defined as:
$$u^n(x,t):=\max\{r^n_i(t)-d_\Omega(x,y^n_i),0\}$$ 
where $r^n_i\in \C^1(0,t_i)\cap \C^0 ([0,+\infty))$  satisfy the following ODE for some $t_i^n>0$:
\begin{equation*}
\left\{
\begin{array}{ll}
\dot r_i^n(t)=\frac{c_i^n}{\vert A_i^n(t)\vert} & \forall t\in ]0,t_i^n[,\\
r_i^n(0)=0, & \\
r_i^n(t)=d_g^+(y_i^n) & \forall t \in [t_i^n, +\infty),\\ 
\end{array}
\right.
\end{equation*}
$$A^n_i(t):=\{ x\in \clomega:\ r^n_i(t)-\vert x-y^n_i\vert=\max_j \{ r_j^n(t)-\vert x-y_j^n\vert, 0\} \}.$$
\end{itemize}
\end{thm}
\begin{proof}
It only remains to show the second part of 3). Indeed, when $u(\cdot, t)< g$, by Lemma \ref{nu-nu}, $\nu_t=0$. So $\gamma_t$ is the solution of the Monge transportation problem with marginals $\partial_tu(\cdot, t)$ and $f$.  As $\partial_t u(\cdot,t)$ is absolutely continuous, so is $\mu_t$ (see \cite{DePPra2002CV,DePEvPRA2004BLMS,DePPr2004COCV,San2009CV}).
\end{proof}
{\textbf{Acknowledgements:}} The second author would like to thank the University of Brest (UBO) and Pisa for financial supports. She is very grateful to the staff of the University of Pisa for the efficiency and warm welcome she always received during her stays in Pisa.

\bibliographystyle{plain}

\end{document}